\documentclass[11pt,a4paper]{article}
\usepackage{amsmath}
\usepackage{amsopn}
\usepackage{amssymb}
\usepackage{latexsym}
\usepackage{amsfonts}
\usepackage[all]{xy}
\usepackage{amsthm}
\usepackage[T1]{fontenc}
\usepackage[utf8]{inputenc}
\usepackage{authblk}

\newtheorem{thm}{Theorem}[section]

\newtheorem{pro}[thm]{Proposition}
\newtheorem{lem}[thm]{Lemma}
\newtheorem{cla}[thm]{Claim}

\newtheorem{cor}[thm]{Corollary}

\theoremstyle{definition}
\newtheorem{obs}[thm]{Observation}

\newtheorem{exa}[thm]{Example}
\newtheorem{defn}[thm]{Definition}

\newtheorem{conj}[thm]{Conjecture}

\newcommand{\een}{\end{enumerate}}
\newcommand{\blem}{\begin{lem}}
\newcommand{\elem}{\end{lem}}
\newcommand{\bcl}{\begin{cla}}
\newcommand{\ecl}{\end{cla}}
\newcommand{\ethm}{\end{thm}}
\newcommand{\bpr}{\begin{pro}}
\newcommand{\epr}{\end{pro}}
\newcommand{\bco}{\begin{cor}}
\newcommand{\eco}{\end{cor}}
\newcommand{\bcon}{\begin{conj}}
\newcommand{\econ}{\end{conj}}
\newcommand{\bde}{\begin{defn}}
\newcommand{\ede}{\end{defn}}
\newcommand{\bex}{\begin{exa}}
\newcommand{\eexa}{\end{exa}}
\newcommand{\bobs}{\begin{obs}}
\newcommand{\eobs}{\end{obs}}
\newcommand{\bexe}{\begin{exe}}
\newcommand{\eexe}{\end{exe}}

\makeatletter
\let\@fnsymbol\@arabic
\makeatother

\begin{document}

\title{Structure of the Group of Balanced Labelings on Graphs, its Subgroups and Quotient Groups.\footnote{An extended abstract of this article has been accepted at the European Conference on Combinatorics, Graph Theory and Applications  - Eurocomb 2013}}
\author[*]{Yonah Cherniavsky \thanks{yonahch@ariel.ac.il}}
\author[**]{Avraham Goldstein \thanks{avraham.goldstein.nyc@gmail.com}}
\author[*]{Vadim E. Levit \thanks{levitv@ariel.ac.il}}
\affil[*]{Ariel University, Israel}
\affil[**]{City University of New York, NY, USA}
\renewcommand\Authands{ and }

\maketitle
\large
\begin{abstract}
We discuss functions from edges and vertices of an undirected graph to an Abelian group. Such functions, when the sum of their values along any cycle is zero, are called balanced labelings. The set of balanced labelings forms an Abelian group. We study the structure of this group and the structure of two closely related to it groups: the subgroup of balanced labelings which consists of functions vanishing on vertices and the corresponding factor-group. This work is completely self-contained, except the algorithm for obtaining the $3$-edge-connected components of an undirected graph, for which we make appropriate references to the literature.

\end{abstract}

\section{Introduction}
Let $G$ be an undirected graph with the set of vertices $V=V(G)$ and the set of edges $E=E(G)$. $G$ does not have to be connected and we permit multiple edges and loops. In the literature such graphs are sometimes called multigraphs or pseudographs. Let $A$ be an Abelian group. In this article we study three topics:
\begin{enumerate}
\item A function $f:E\rightarrow A$ is called balanced if the sum of its values along any cycle of $G$ is zero. By a cycle we mean a closed truncated trail (the definition is given in the beginning of Section~\ref{Prl}). The set $H(E,A)$ of all balanced functions $f:E\rightarrow A$ is a subgroup of the free Abelian group $A^E$ of all functions from $E$ to $A$. We give a full description of the structure of the group $H(E,A)$ and provide an $O(|E|)$-time algorithm to construct a set of the generators of its direct summands which are the group $A$ and the subgroup of involutions of $A$ which we denote $A_2$. In the literature the values of functions from $E$ to $A$ are called weights.
\item A function $g:V\rightarrow A$ is called balanceable if there exists some $f:E\rightarrow A$ such that the sum of all values of $g$ and $f$ along any cycle of $G$ is zero. Since our trails are truncated, the final vertex is not present in the trail, so the summation stops at the last edge and, thus, does not include the value on the last vertex. The set $B(V,A)$ of all balanceable functions $g:V\rightarrow A$ is a subgroup of the free Abelian group $A^V$ of all functions from $V$ to $A$. We give a full description of the structure of the group $B(V,A)$.
\item A function $h:V\bigcup E\rightarrow A$, which takes values on both vertices and edges of $G$, is called balanced if the sum of its values along any cycle of $G$ is zero. As before, due to the truncation of our trails, the value on the last vertex is not included. The set $W(V\bigcup E,A)$ of all balanced functions $h:V\bigcup E\rightarrow A$ is a subgroup of the free Abelian group $A^{V\bigcup E}$ of all functions from $V\bigcup E$ to $A$. The group $H(E,A)$ is naturally isomorphic to the subgroup of $W(V\bigcup E,A)$, consisting of all functions which take every vertex of $G$ to $0$. Abusing the notation, we identify $H(E,A)$ and that subgroup of $W(V\bigcup E,A)$. We show that $B(V,A)$ is naturally isomorphic to the factor group $W(V\bigcup E,A)/H(E,A)$ and use this fact to give a full description of the group $W(V\bigcup E,A)$. In the literature, for example in~\cite{NI}, the functions from $V\bigcup E$ to $A$ are called labelings.
\end{enumerate}
The study of the cycle space of a directed graph and of the integer-valued edge functions vanishing on all cycles of that graph, which are dually related to the the cycle space, are classical in graph theory. First time the similar questions were considered for the connected undirected graphs by Balakrishnan and Sudharsanam in~\cite{OI}. In this work the functions from edges to real numbers, which vanish on the cycles, are called cycle-vanishing edge valuations. It turns out that for the undirected case the dimension of the cycle space and, dually, the dimension of the space of cycle vanishing edge valuations are closely related to another classical concept in graph theory $-$ that of $3$-edge connectivity in an undirected graph. In a very recent work~\cite{NI}, Joglekar, Shah and Diwan consider functions, which take values on both vertices and edges of a connected undirected graph, with their values in a finite Abelian group $A$. In that work, which is closely related to~\cite{OI}, the amount of such functions is calculated among other important properties of these functions.
\\ \\
In this article we find the group structure of the Abelian groups $H(E,A)$, $B(V,A)$ and $W(V\bigcup E,A)$ for a finite undirected graph $G$. With this motivation in mind we review some edge-connectivity properties of undirected graphs. We discuss the $k$-edge-connectivity, present the edge version of the Menger's Theorem and provide an alternative proof. We state and prove several notions and results, which are commonly used in the literature, but usually neither formally stated nor proved there. Further, we investigate $3$-edge-connectivity related properties of graphs. This leads us several new findings concerning the $3$-edge-connectivity structure of a graph.

Additionally, we refer to a recent very efficient algorithm, which computes the $3$-edge-connected components of an undirected graph.

In parallel we recall and discuss the $\mathbb{F}_2$ homology theory of undirected graphs and make use of some of its results for our studies.

Finally, we combine all of the above to describe the groups $H(E,A)$, $B(V,A)$ and $W(V\bigcup E,A)$. Except for the above-mentioned algorithm, some basic properties of Abelian groups, and some basic notions from Linear Algebra, the material of this article is self-contained. For additional detailed exposition of the topics, presented in this article, we refer to~\cite{Diestel} and~\cite{Harju}. In~\cite{RZ} the triples $(\Gamma, g, G)$, where $\Gamma$ is a not necessary finite undirected graph, $G$ is a group and $g$ is a function from the set of edges of $\Gamma$ to $G$, are considered. The edges of $\Gamma$ are thought of as having an arbitrary, but fixed, orientation and the equality $g(-e)=(g(e))^{-1}$ holds for every edge $e$, where $-e$ is $e$, with its orientation inverted. Such triples are called gain graphs or voltage graphs and the values of $g$ on the edges are called gains. Gain graphs are called balanced if $g$ along every cycle of $\Gamma$ is the identity element. For the voltage graphs being balanced is the same as satisfying the Kirchhoff's voltage law. In~\cite{RZ} gain graphs with an Abelian group $G$ are studied and several strong criteria for such gain graphs to be balanced are obtained. Voltage graphs are discussed in~\cite{GT}. For the survey of bibliography on signed and gain graphs and allied areas see~\cite{Zaslavsky}.

\section{Preliminaries}\label{Prl}
An undirected graph $G$ consists of a set of vertices $V=V(G)$ and a set of edges $E=E(G)$. To each $e\in E$ corresponds a pair of vertices $v,w\in V$. Abusing the terminology, for any vertex $v$ we permit the pair $v,v$ to correspond to a loop edge $e$. We say that the vertices $v$ and $w$ are connected by the edge $e$ and are adjacent to $e$. We also say that the edge $e$ goes between $v$ and $w$ and that the edge $e$ is adjacent to the vertices $v$ and $w$.
\bde A trail from a vertex $x$ to a vertex $y$ is an alternating sequence $v_1,e_1,v_2,e_2,...,v_{n},e_{n},v_{n+1}$ of vertices and edges such that $v_1=x$, $v_{n+1}=y$, each $e_j$, for $j=1,...,n$, goes between $v_j$ and $v_{j+1}$, and $e_i\ne e_j$ for $i\ne j$.
\ede
Note: The last part of our definition asserts that a trail does not have any repeating edges.
\bde A truncated trail (\textit{ttrail}) $p$ from a vertex $x$ to a vertex $y$ is an alternating sequence $v_1,e_1,v_2,e_2,...,v_{n},e_{n}$ of vertices and edges such that $v_1,e_1,v_2,e_2,...,v_{n},e_{n},y$ is a trail from $x$ to $y$.
\ede
\bde Let $p=v_1,e_1,v_2,e_2,...,v_{n},e_{n}$ be a ttrail from a vertex $x$ to a vertex $y$ and $p'=v'_1,e'_1,v'_2,e'_2,...,v'_{n'},e'_{n'}$ be a ttrail from a vertex $y$ to a vertex $w$. Then the alternating sequence $p+p'$ of vertices and edges is defined as $v_1,e_1,v_2,e_2,...,v_{n},e_{n},v'_1,e'_1,v'_2,e'_2,...,v'_{n'},e'_{n'}$.
\ede
The alternating sequence $p+p'$ will itself be a ttrail going from $x$ to $w$ if and only if the ttrails $p$ and $p'$ do not have any common edges.
\bde Let $p=v_1,e_1,v_2,e_2,...,v_{n},e_{n}$ be a ttrail from a vertex $x$ to a vertex $y$. Then the ttrail $p^{-1}$ from $y$ to $x$ is defined as $v'_1,e'_1,v'_2,e'_2,...,v'_{n},e'_{n}$ where $v'_1=y$, $v'_i=v_{n+2-i}$ for $i=2,...,n$, and $e'_i=e_{n+1-i}$ for $i=1,...,n$.
\ede
Note: The edges and the inner vertices of $p^{-1}$ are just the edges and the inner vertices of $p$, taken in the reversed order.\\
Whenever we do not care about the direction of a ttrail and only make reference to its edges, we will abuse the notations and speak of $p$ and of $p^{-1}$ as ttrails between the vertices $x$ and $y$.
\bde A ttrail $p$ from a vertex $x$ to itself is called a closed ttrail.
\ede
Note: We permit the trivial closed ttrail, which contains no vertices and no edges.
\bde A closed ttrail $p=v_1,e_1,v_2,e_2,...,v_{n},e_{n}$ is called a cycle if $v_i\ne v_j$ whenever $i\ne j$.
\ede
Let $A$ be an Abelian group with the group operation denoted by $+$ and the identity element denoted by $0$.
\bde A function $f:E\rightarrow A$ is called balanced if the sum $f(e_1)+...+f(e_n)$ of the values of $f$ on all edges of each closed ttrail of $G$ is equal to $0$.
\ede
\bde The set of all balanced functions $f:E\rightarrow A$ is denoted by $H(E,A)$. $H(E,A)$ is a subgroup of the Abelian group $A^E$ of all functions from $E$ to $A$.
\ede
\bde A function $g:V\rightarrow A$ is called balanceable if there exists some $f:E\rightarrow A$ such that the sum of all values $g(v_1)+f(e_1)+g(v_2)+f(e_2)+...+g(v_n)+f(e_n)$ along every closed ttrail of $G$ is zero. We say that this function $f:E\rightarrow A$ balances the function $g:V\rightarrow A$.
\ede
\bde The set of all balanceable functions $g:V\rightarrow A$ is denoted by $B(V,A)$. $B(V,A)$ is a subgroup of the free Abelian group $A^V$ of all functions from $V$ to $A$.
\ede
\bde A function $h:V\bigcup E\rightarrow A$, which takes both vertices and edges of $G$ to some elements of $A$, is called balanced if the sum of its values $h(v_1)+h(e_1)+h(v_2)+h(e_2)+...+h(v_n)+h(e_n)$ along each closed ttrail of $G$ is zero.
\ede
\bde The set of all balanced functions $h:V\bigcup E\rightarrow A$ is denoted by $W(V\bigcup E,A)$. $W(V\bigcup E,A)$ is a subgroup of the Abelian group $A^{V\bigcup E}$ of all functions from $V\bigcup E$ to $A$.
\ede
Clearly, any balanced function $f:E\rightarrow A$ can be viewed as a balanced function from $V\bigcup E$ to $A$, which is $0$ on every vertex of $G$. Thus, we regard $H(E,A)$ as a subgroup of $W(V\bigcup E,A)$. Furthermore, these three groups are related as follows.
\begin{lem} \label{quotient.group} $H(E,A)$ is a subgroup of $W(V\bigcup E,A)$ and the quotient\\ $W(V\bigcup E,A)/H(E,A)$ is naturally isomorphic to $B(V,A)$.
\end{lem}
\begin{proof} The natural isomorphism from $W(V\bigcup E,A)$ onto $B(V,A)$ ``forgets'' the values of the function $h\in W(V\bigcup E,A)$ on the edges of $G$ and thus ``regards'' it as just a balanceable function from $V$ to $A$. The kernel of this isomorphism is the subgroup of $W(V\bigcup E,A)$ consisting of all functions $h$ such that $h$ is $0$ on every vertex of $G$. This subgroup is precisely $H(E,A)$.
\end{proof}
We review some basic definitions and facts concerning Abelian groups:
\bde The order $ord(a)$ of an element $a\in A$ is the minimal positive number such that $ord(a)a=0$. If no such positive number exist we say that $ord(a)=\infty$.
\ede
\bde The set of all elements of $A$ of order $2$ is denoted by $A_2$. $A_2$ is a subgroup of $A$.
\ede
\bde The image of the doubling map from $A$ to itself, which multiplies every element of $A$ by $2$, is a subgroup of $A$ denoted by $2A$.
\ede
Note: Any element $b \in 2A$ has at least one element $a\in A$ such that $b=2a$. This $a$ is called a half of $b$. There could be another element $a'\in A$ such that $b=2a'$. However $a-a'$ must belong to $A_2$ since $2(a-a')=b-b=0$. Vice versa, adding any $c$ from $A_2$ to $a$ again produces a half of $b$. Thus, the half of an element from $2A$ in $A$ is defined up to an element of $A_2$. Indeed, $A_2$ is precisely the kernel of the doubling map.
\\ \\
Now let $\mathbb{F}_2$ be a field with two elements $0$ and $1$. Vector spaces $\mathbb{F}_2^E$ and $\mathbb{F}_2^V$ are defined as the sets of all binary sequences whose entries correspond, respectively, to the edges and the vertices of $G$. There is a trivial one-to-one correspondence between all the subsets of $E$ and all the elements of $\mathbb{F}_2^E$ and a trivial one-to-one correspondence between all the subsets of $V$ and all the elements of $\mathbb{F}_2^V$. Abusing the notation, we speak interchangeably of sets of edges and of vectors in $\mathbb{F}_2^E$. Similarly, we speak interchangeably of sets of vertices and of vectors in $\mathbb{F}_2^V$.
\bde The boundary linear map $\delta:\mathbb{F}_2^E\rightarrow \mathbb{F}_2^V$ is defined by taking each edge to the sum of its two adjacent vertices.
\ede
\bde The kernel $Ker(\delta)$ of $\delta$ is called the Cycle Space of $G$ and its elements are called the homological cycles of $G$.
\ede
For each ttrail $p$ let $\epsilon(p)\in \mathbb{F}_2^E$ be the set of all edges which are present in $p$. It is easy to see that if $p$ is a closed ttrail then $\epsilon(p)$ is a homological cycle. Also, it is a well-known result that:
\begin{lem} \label{homologicalcyclesum} Every homological cycle $c$ is a sum $\epsilon(c_1)+...+\epsilon(c_n)$ where $c_1,...,c_n$ are some cycles such that every edge of $c$ appears in at least one of the $c_1,...,c_n$.\end{lem}
\begin{proof} We prove it by induction on the number $m$ of edges in $c$. If $m=0$ then $c$ is the trivial homological cycle and $c=\epsilon(c_1)$ where $c_1$ is the trivial ttrail. If $m=1$ then $c=\{e\}$, where $e$ is a loop. Then the sequence $c_1=v,e$, where $v$ is the vertex adjacent to $e$, is a cycle and $c=\epsilon(c_1)$. Assume that our lemma holds for all homological cycles containing up to $m$ edges. Let $c$ be a homological cycle containing $m+1$ edges. We call these edges $e_1,...,e_{m+1}$ and their adjacent vertices $v_1,w_1; ...; v_{m+1},w_{m+1}$, respectively.
\\ \\
If $e_1$ is a loop (so $v_1=w_1$) then we define the cycle $c_1$ to be sequence $v_1,e_1$. Since $\delta(e_2+...+e_{m+1})=v_1+w_1=0$, we get that $e_2+...+e_{m+1}$ is a homological cycle with $m$ edges and, by the induction hypothesis, $e_2+...+e_{m+1}=\epsilon(c_2)+...+\epsilon(c_n)$ for some cycles $c_2,...,c_n$. Thus $c=\epsilon(c_1)+...+\epsilon(c_n)$.
\\ \\
If $e_1$ is not a loop, which means $v_1\ne w_1$, then there must be at least one non-loop edge in $c$, different from $e_1$, which is adjacent to the vertex $w_1$. We call this edge $e_2$ and its vertex $w_1$ we name $v_2$. If $w_2$ is equal to $v_1$ then $v_1,e_1,v_2,e_2$ is a cycle and $e_3+...+e_{m+1}$ is a homological cycle. In that case the lemma follows by applying the induction hypothesis. If $w_2\ne v_1$ then there must be at least one non-loop edge in $c$, different from $e_1$ and $e_2$, which is adjacent to the vertex $w_2$. We call this edge $e_3$. This process goes on, until for some $1\le i<j\le m+1$, $w_j$ is equal to $v_i$. Then $c_1=v_i,e_i,...,v_j,e_j$ is a cycle and ``the remaining'' sum $e_1+...+e_{i-1}+e_{j+1}+...+e_{m+1}$ is a homological cycle, to which the induction hypothesis applies.
\end{proof}
\bde For two vertices $v, w$ of $G$, a homological path $p$ between $v$ and $w$ is a vector $p\in \mathbb{F}_2^E$ such that $\delta(p)=v+w$
\ede
Thus, a homological cycle is a homological path between any vertex $v$ and itself.
\\ \\
We conclude this section by stating and proving a classical result, but in a slightly more general form with comparison to its usual presentation in the literature, since our graphs can have multiple edges, loops and do not need to be connected:
\begin{thm} \label{classicaldimension} The dimension $dc(G)$ of the Cycle Space of a graph $G$ is equal to $|E|-|V|+con(G)$.
\end{thm}
\begin{proof} This theorem is evident for an empty graph. We proceed to prove it for non-empty graphs by induction on $|E|$. If $|E|=0$ then $G$ consists of $con(G)$ vertices and zero edges and, thus, has zero cycles. So our theorem holds. Assume that the theorem holds for all graphs with $|E|\le n$. Any graph $G'$ with $|E|=n+1$ can be obtained from some graph $G$ with $|E|=n$ by attaching to it an edge $e_{n+1}$. If $e_{n+1}$ connects vertices in two different connected components of $G$ then $dc(G')=dc(G)$ and $con(G')=con(G)-1$. So the theorem holds.
\\ \\
If $e_{n+1}$ connects vertices $x$ and $y$ which belong to the same connected component of $G$ then the sum of $e_{n+1}$ and any homological path $p$ in $G$ between $x$ and $y$ is an element of the Cycle Space of $G'$, but not of the Cycle Space of $G$. On the other hand, any element of the Cycle Space of $G'$, which is not an element of the Cycle Space of $G$, must contain $e_{n+1}$. Thus, in this case $dc(G')=dc(G)+1$ and $con(G')=con(G)$. So, again, the theorem holds.
\end{proof}

\section{$k$-Edge-Connectivity}
\bde Two distinct vertices $v$ and $w$ of $G$ are called $k$-edge-connected if exist $k$ ttrails $p_1,...,p_k$ between $v$ and $w$ such that any $p_i$ and $p_j$, where $i\ne j$, have no common edges. By definition, we say that every vertex $v$ of $G$ is $k$-edge-connected to itself for any positive integer $k$.
\ede
We say that the vertices $v$ and $w$ are connected if they are $1$-edge-connected. Otherwise, we say that $v$ and $w$ are disconnected. All the vertices connected to $v$, together with all the edges adjacent to them, are called the connected component $Con(v)$ of $v$.
\\\\
\textbf{Menger's Theorem.} \textit{Two distinct vertices $v$ and $w$ are $k$-edge-connected if and only if no deletion of any $k-1$ edges from $G$ disconnects $v$ and $w$.}
\begin{proof} In one direction this theorem is trivial $-$ if $p_1,...,p_k$ are pairwise edge-disjoint ttrails between $v$ and $w$ then any deletion of any $k-1$ edges $e_1,...,e_{k-1}$ from $G$ will not cut all these $k$ ttrails. So $x$ and $y$ will stay connected after any such deletion. Now we prove the other direction. Assume that the statement is false. Then there exist counter-examples to it. Let $G$ be a counter-example with the minimal amount of edges in it. Thus, $G$ is an undirected graph containing some (distinct) vertices $v$ and $w$, which can not be connected in $G$ by $k$ pairwise edge-disjoint ttrails, such that a deletion of any $k-1$ edges form $G$ will not disconnect $v$ and $w$.
\\ \\
Due to the minimality of $G$, every edge $e$ of $G$ must belong to some (at least one) collection of $k$ different edges, deletion of all of which from $G$ disconnects $v$ and $w$ (otherwise deleting $e$ from $G$ produces a counter-example with less edges than $G$).
\\ \\
If there exists a collection $e_1,...,e_k$ of edges of $G$, such that after its deletion from $G$ vertices $v$ and $w$ become disconnected and the connected components $Con(v)$ of $v$ and $Con(w)$ of $w$ still both contain some edges, then denote by $v_1,...,v_k$ the adjacent vertices of the deleted $e_1,...,e_k$ which belong to $Con(v)$ and by $w_1,...,w_k$ their adjacent vertices which belong to $Con(w)$. Some of the vertices $v_1,...,v_k$ might coincide. Some of the vertices $w_1,...,w_k$ might also coincide. Construct graph $G'$ by taking $Con(v)$ and a new vertex $v'$ and connecting $v'$ by new edges $e''_1,...,e''_k$ to the vertices $w_1,...,w_k$ in $Con(w)$. Since each $G'$ and $G''$ has less edges than $G$ and since any deletion of any $k-1$ edges from one of them will not disconnect neither $v$ from $v'$ nor $w'$ from $w$, we get that $v$ is $k$-edge-connected to $v'$ in $G'$ and $w'$ is $k$-edge-connected to $w$ in $G''$. Thus, there exist $k$ pairwise edge-disjoint ttrails $p'_i$ going from $v$ to $v_i$, for $i=1,...k$, and $k$ pairwise edge-disjoint ttrails $p''_i$ going from $w_i$ to $w$, for $i=1,...k$. Thus, the ttrails $p_1,...,p_k$, where each $p_i$ is $p'_i,v_i,e_i,p''_i$, are pairwise edge-disjoint ttrails from $v$ to $w$. This contradicts our assumption that $v$ and $w$ can not be connected in $G$ by $k$ pairwise edge-disjoint ttrails.
\\ \\
Thus, the deletion of any collection $e_1,...,e_k$ of edges of $G$, which disconnects $v$ and $w$ must leave either $Con(v)$ or $Con(w)$ (or both) with no edges. But if there exists any edge $e_1$ in $G$, which is not adjacent neither to $v$ nor to $w$, then, as discussed above, there exists some collection $e_1,e_2,...,e_k$ of $k$ edges, deleting which disconnects $v$ from $w$. But deleting $e_2,...,e_k$ does not disconnect $v$ from $w$. Hence deleting $e_1,e_2,...,e_k$ leaves some edges in both $Con(v)$ and $Con(w)$, which is impossible. Thus, every edge of a minimal counterexample $G$ is either going between $v$ and some vertex $v_i$ or between $w$ and some vertex $w_j$.
\\ \\
Let $u_1,...,u_h$ be all the vertices of $G$, which are different from $v$ and $w$. Let $e_1,...,e_a$ be the edges connecting $v$ to $w$ (if such are present in $G$) and $e'_{i,1},...,e'_{i,b_i}$ and $e''_{i,1},...,e''_{i,c_i}$ be the edges connecting each $u_i$ with $v$ and $w$ respectively. If $b_i\ne c_i$ for some $i$ then deleting an edge between $u_i$ and $v$ (if $b_i>c_i$) or between $u_i$ and $e$ (if $b_i<c_i$) will prodice a counter-example with less edges than $G$, which contradicts the minimality of $G$. Thus, the number of edges between $v$ and each $u_i$ must be equal to the number of edges between that $u_i$ and $w$. But this implies that we have $k$ pairwise edge-disjoint ttrails from $v$ to $w$, which contradicts the fact that $G$ is a counter-example. \end{proof}
For the other proofs of Menger's Theorem see for example \cite{Menger} or \cite{Harju}.
\\ \\
The following well-known fact appears a lot in the literature, but usually no proof is provided:
\begin{lem} $k$-edge-connectivity is an equivalence relation between the vertices of $G$.
\end{lem}
\begin{proof} By definition, every vertex is $k$-edge-connected to itself. Clearly, if $v$ is $k$-edge-connected to $w$ then $w$ is $k$-edge-connected to $v$. We need to show that if $v$ is $k$-edge-connected to $w$ and $w$ is $k$-edge-connected to $u$ then $v$ is $k$-edge-connected to $u$. But, since no removal of any $k-1$ edges disconnects $v$ from $w$ nor $w$ from $u$, no removal of $k-1$ edges can disconnect $v$ from $u$.
\end{proof}
Thus $k$-edge-connectivity, for any positive integer $k$, breaks the set $V$ of $G$ into equivalence classes of $k$-edge-connected vertices.
\bde We denote the number of equivalence classes of $k$-edge-connected vertices of $G$ by $con_k(G)$.
\ede
\bde For a vertex $v\in V$ we denote the equivalence class of $k$-edge-connected vertices of $G$, which contains $v$, by $Con_k(v)$.
\ede
Note that $con_1(G)$ is exactly equal to the number $con(G)$ of the connected components of $G$. Note that $Con_1(v)$ is exactly the set of all vertices in $Con(v)$.
\bde A ttrail between two $k$-edge-connected vertices is called a $k$-weakly closed ttrail.
\ede
The rational behind this terminology becomes clear in the proof of Theorem \ref{weak.cycle} and in \cite{OI}. Every closed ttrail of $G$ is also a $k$-weakly closed ttrail of $G$ for every integer $k$. Conversely, since $G$ has a finite amount of edges, there exists some $k\le |E|+1$ such that any $k$-weakly closed ttrail of $G$ will also be a closed ttrail of $G$.
\bde A homological path between two $k$-edge-connected vertices of $G$ is called a $k$-weak homological cycle.
\ede
\bde The subspace of $\mathbb{F}_2^E$ spanned by all $k$-weak homological cycles of $G$ is called the $k$-Weak Cycle Space of $G$.
\ede
Clearly, the Cycle Space of $G$ is a subspace of the $k$-Weak Cycle Space of $G$ for every integer $k$.
\begin{thm} \label{weak.cycle} The dimension $dc_k(G)$ of the $k$-Weak Cycle Space of a graph $G$ is equal to $|E|-con_k(G)+con(G)$.
\end{thm}
\begin{proof} Let $G'$ be the graph, obtained from $G$ by identifying (gluing together) all the vertices in each $k$-edge-connectivity equivalence class of $G$. Then $G'$ has $con_k(G)$ vertices, $|E|$ edges, and $con(G)$ connected components. Clearly, each $k$-weakly closed ttrail of $G$ becomes a closed ttrail of $G'$ and, conversely, any pre-image of every closed ttrail of $G'$ in $G$ is a $k$-weakly closed ttrail of $G$. Applying Theorem \ref{classicaldimension} to $G'$ yields our theorem.
\end{proof}
\begin{lem} \label{weak.hom.decomp} Every $k$-weak homological cycle $c$ of $G$ is a sum $\epsilon(p_1)+...+\epsilon(p_m)$, where each $p_j=v_{j,1},e_{j,1},...,v_{j,t_j},e_{j,t_j}$ is a $k$-weakly closed ttrail from some vertex $w_j$ to some (not necessarily different) vertex $w'_j$, such that no two of its vertices $v_{j,a}$ and $v_{j,b}$, where $a\ne b$, are $k$-edge-connected.
\end{lem}
\begin{proof} Again, let $G'$ be the graph, obtained from $G$ by identifying (gluing together) all the vertices in each $k$-edge-connectivity equivalence class of vertices of $G$. This gluing induces a bijection between the $k$-weak homological cycles of $G$ and the homological cycles of $G'$. It also induces a bijection between $k$-weakly closed ttrails of $G$, which satisfy the condition that no two vertices of a ttrail are $k$-edge-connected, and the cycles of $G'$. By Lemma \ref{homologicalcyclesum} every homological cycle $c'$ of $G'$ is a sum $\epsilon(c'_1)+...+\epsilon(c'_n)$, where $c'_1, ..., c'_n$ are some cycles of $G'$. Thus, every $k$-weak homological cycle $c$ of $G$ is a sum $\epsilon(p_1)+...+\epsilon(p_m)$ of ttrails $p_1, ..., p_m$, each of which goes between $k$-edge-connected vertices of $G$ and contains no two $k$-edge-connected vertices.
\end{proof}
Let $v_1,...,v_k\in V$ be (not necessarily pairwise distinct) vertices, belonging to the same $k$-edge-connected component of $V$.
\begin{lem} \label{ext.menger} If there exist $k$ pairwise edge-disjoint ttrails between a vertex $u$ and the vertices $v_1,...,v_k$ then $u$ is $k$-edge-connected to $v_1$.
\end{lem}
\begin{proof} Clearly, no deletion of $k-1$ edges can disconnect $u$ from all the vertices $v_1,..,v_k$. By Menger's Theorem no deletion on $k-1$ edges can disconnect $v_1$ from any vertex of $v_1,..,v_k$. Thus, no deletion of $k-1$ edges can disconnect $u$ from $v_1$, which, by Menger's Theorem, implies that $u$ is $k$-edge-connected to $v_1$.\end{proof}
In this article we are interested in the case when $k=3$, because the structures and the results, obtained for that case, will provide the answers for the three main questions on this article.
\begin{thm} There exists an algorithm, linear in time and space with respect to $max(|E|,|V|)$, which finds all the $3$-edge-connected
components of $G$.
\end{thm}
\begin{proof} We start by deleting all the loop edges of $G$, since the loop edges do not contribute to connectivity between different vertices. Next, we apply Tsin's algorithm from \cite{Tsin}. This optimal algorithm for $3$-edge-connectivity first executes another simple, linear time and linear space algorithm to remove all bridge edges - the edges whose removal increases the amount of connected components of $G$. Next it performs one pass over $G$ to determine a set of cut-pairs, whose removal leads to the 3-edge-connected components. An additional, final, pass determines all $3$-edge-connected components of $G$. Tsin's algorithm is simple, easy to implement and runs in linear time and space.
\end{proof}
\bde A $3$-weakly closed ttrail between vertices $v$ and $w$, where $v$ can be equal to $w$, is called short if none of its inner vertices are $3$-edge-connected to $v$.
\ede
Note, that by our definition, the trivial empty closed ttrail is short.
\begin{lem} \label{weak.space.dec} Every non-trivial $3$-weakly closed ttrail $c$ is a sum $c_1+...+c_m$ of some non-trivial short $3$-weakly closed ttrails $c_1,...,c_m$ and this decomposition is unique.
\end{lem}
\begin{proof} If $c$ is a $3$-weakly closed ttrail consisting only of one vertex and one edge then $c$ is obviously a short $3$-weakly closed ttrail. We proceed by induction on number of edges in $c$. Suppose that the lemma holds for all $3$-weakly closed ttrails with less than $m$ edges. Let $c=v_1,e_1,...,v_m,e_m$ be a $3$-weakly closed ttrail. If $c$ is short then we are done. Otherwise, let $i$, where $1<i$, be the smallest index such that $v_i$ is $3$-edge-connected to $v_1$. Define $c_1=v_1,e_1,...,v_{i-1},e_{i-1}$ and $c'=v_i,e_i,...,v_m,e_m$. Clearly, $c_1$ is a short $3$-weakly closed ttrail and $c=c_1+c'$. Now we apply the induction hypothesis to obtain the unique decomposition of $c'$ into the sum of short $3$-weakly closed ttrails. This gives us a decomposition of $c$ into a sum of short $3$-weakly closed ttrails and the uniqueness of this decomposition is obvious.
\end{proof}
Let $x,y,u,w$ be any (not necessarily distinct) vertices in the same $3$-edge-connected component of $V$.
\begin{thm} \label{theorem.twin} Let $c=x,e_1,...,v_h,e_h$ and $c'=u,e'_1,...,v'_t,e'_t$ be two short $3$-weakly closed ttrails,  $c$ goes from the vertex $x$ to the vertex $y$, $c'$ goes from the vertex $u$ to the vertex $w$. If $c$ and $c'$ have a common inner vertex, then $x=u$, $e_1=e'_1$, $y=w$, and $e_h=e'_t$.
\end{thm}
\begin{proof} Let $v_i=v'_j$ be the common inner vertex of $c$ and $c'$ with the smallest index $i$ amongst all such inner vertices $v_2,...,v_h$ of $c$. Then, unless $i=2$ and $e_1=e'_1$, three paths $p_1=x,e_1,...,v_{i-1},e_{i-1}$, $p_2=u,e'_1,...,v'_{j-1},e'_{j-1}$ and $p_3=v'_j,e'_j,...,v'_t,e'_t$ are pairwise edge-disjoint and connect between $v_i$ and the vertices $x,u,w$. Thus, by Lemma \ref{ext.menger}, $v_i$ is $3$-edge-connected to $x$, which contradicts the fact that $c$ is short. Since $e_1=e'_1$ we get that $x=u$. The same argument shows that $e_h=e'_t$ and $y=w$.
\end{proof}
\bde Short $3$-weakly closed ttrails are called twins if they have the same first and last edges.
\ede
Clearly, being twins defines an equivalence relation between the short $3$-weakly closed ttrails connecting vertices of the same $3$-edge-connected component of $V$.
\section{Balanced Functions from the Set of Edges $E$ to an Abelian Group $A$}
Let $G$ be an undirected graph and $A$ be an Abelian group.
\begin{thm} \label{four.one} If a function $f:E\rightarrow A$ takes the Cycle Space of $G$ to $0$ then $f$ is balanced. Conversely, if a function $f:E\rightarrow A$ is balanced then $f$ takes the $3$-Weak Cycle Space of $G$ into $A_2$ and takes the Cycle Space of $G$ to $0$.
\end{thm}
\begin{proof} Suppose that $f$ takes every homological cycle of $G$ to $0$.
Then for any closed ttrail $c=v_1,e_1,...,c_m,e_m$ of $G$ all its edges are distinct, so we have $$f(e_1)+...+f(e_m)=f(\epsilon(c))=0.$$ Thus, $f$ is balanced.
\\ \\
Due to Lemma \ref{weak.space.dec}, $f$ takes the $3$-Weak Cycle Space of $G$ into $A_2$ if and only if $f$ takes all short $3$-weak homological cycles of $G$ to some elements of $A_2$. Thus, to prove the second half of our theorem we consider only short $3$-weak homological cycles.
\\ \\
Suppose that $f$ is balanced. From Lemma \ref{homologicalcyclesum} it follows that $f$ takes every homological cycle of $G$ to $0$. Let $c$ be any short $3$-weakly closed ttrail of $G$, which goes between some vertices $v$ and $w$. Then there exist three pairwise edge-disjoint ttrails $p_1, p_2$ and $p_3$ from $v$ to $w$. If $c$ has a common edge with $p_3$ then, due to Lemma \ref{weak.space.dec} and Theorem \ref{theorem.twin}, $c$ must have a common first and last edges with one of the short $3$-weakly closed ttrails from the decomposition of $p_3$. This implies that $c$ can not have any common edges with $p_1$ or $p_2$. So $c$ can have common edges with at most one out of the three ttrails $p_1, p_2$ and $p_3$. Suppose $c$ does not have any common edges with $p_1$ and $p_2$. Then we obtain $f(c+p_1)=f(c+p_2)=f(p_1+p_2)=0$. So $f(c)+f(p_1)=f(c)+f(p_2)=f(p_1)+f(p_2)=0$.
Hence, $2f(c)=[f(c)+f(p_1)]+[f(c)+f(p_2)]-[f(p_1)+f(p_2))=0$.
So $f(\epsilon(c))\in A_2$. It now follows from Lemma \ref{weak.hom.decomp}, for $k=3$, that $f$ takes every $3$-weak homological cycle to an element of $A_2$.
\end{proof}
\begin{thm} \label{balanced.edge} The Abelian group $H(E,A)$ of all balanced functions from $E$ to $A$ is isomorphic to $A^{con_3(G)-con(G)}\times A_2^{|V|-con_3(G)}$
\end{thm}
\begin{proof} Select and fix a basis $c_1,c_2,...$ of the Cycle Space of $G$. Next, select and fix some $3$-weak homological cycles $c'_1,c'_2,...$ of $G$ so that $c_1,c_2,...,c'_1,c'_2,...$ is a basis of the $3$-Weak Cycle Space of $G$. Denote the subspace of the $3$-Weak Cycle Space of $G$, which is spanned by $c'_1,c'_2,...$, by $\Lambda$. Notice that the intersection of $\Lambda$ and the Cycle Space of $G$ contains only the zero vector. From Theorems \ref{classicaldimension} and \ref{weak.cycle} it follows that the basis $c_1,c_2,...,c'_1,c'_2,...$ consists of $|E|-|V|+con(G)$ vectors $c_1,...,c_{|E|-|V|+con(G)}$ and of $(|E|-con_3(G)+con(G))-(|E|-|V|+con(G))=|V|-con_3(G)$ vectors $c'_1,...,c'_{|V|-con_3(G)}$.
\\
Next, select and fix some edges $e_1,...,e_{con_3(G)-con(G)}$ from $E$ so that $c_1,c_2,...,c_{|E|-|V|+con(G)},c'_1,c'_2,...,c'_{|V|-con_3(G)},e_1,...,e_{con_3(G)-con(G)}$ is a basis of $\mathbb{F}_2^E$. Finally, define a homomorphism $$\phi:H(E,A)\rightarrow A^{con_3(G)-con(G)}\times A_2^{|V|-con_3(G)}$$ by
$$\phi(f)=(f(e_1),...,f(e_{con_3(G)-con(G)}),f(c'_1),...,f(c'_{|V|-con_3(G)})).$$
Theorem \ref{balanced.edge} assures that $\phi$ maps $H(E,A)$ to $A^{con_3(G)-con(G)}\times A_2^{|V|-con_3(G)}$.
\\ \\
If $\phi(f)=0$ then $f$, as a balanced function, must be zero on every $c_1,...,c_{|E|-|V|+con(G)}$ and also must be zero on every $c'_1,...,c'_{|V|-con_3(G)}$ and every $e_1,...,e_{con_3(G)-con(G)}$, by the definition of $\phi$. So $f$ must be zero on every element of $\mathbb{F}_2^E$. So $f$ is the zero function. Thus, $\phi$ is one-to-one.
\\ \\
Now we show that $\phi$ is onto. For every element $$\textbf{a}=(a_1,...,a_{con_3(G)-con(G)},a'_1,...,a'_{|V|-con_3(G)})$$ of $A^{con_3(G)-con(G)}\times A_2^{|V|-con_3(G)}$ we construct $f\in H(E,A)$ such that $\phi(f)=\textbf{a}$ by extending
\begin{enumerate}
\item $f(e_t)=a_t$ for all $e_1,...,e_{con_3(G)-con(G)}$,\\
\item $f(c'_j)=a'_{j}$ for all $c'_1,...,c'_{|V|-con_3(G)}$,\\
\item $f(c_i)=0$ for all $c_1,...,c_{|E|-|V|+con(G)}$
\end{enumerate}
$\mathbb{F}_2$-linearly to the entire $\mathbb{F}_2^E$. Requirements (2) and (3), together with Theorem \ref{balanced.edge}, assure that $f$ is balanced. A straightforward plugging-in shows that $\phi(f)=\textbf{a}$.
\end{proof}
\section{Balanceable Functions from the Set of Vertices $V$ to an Abelian Group $A$}
\begin{lem} \label{if.balanceable} If a function $g:V\rightarrow A$ is balanceable then for any $3$-edge-connected vertices $v$ and $w$ of $G$, $g(v)-g(w)\in 2A$
\end{lem}
\begin{proof} If $v=w$ then the lemma is trivial. Otherwise, let $f:E\rightarrow A$ be a function that balances $g$. Let $p_1, p_2$ and $p_3$ be three pairwise edge-disjoined paths from $v$ to $w$. Then the sums of values of $g$ and $f$ along closed ttrails $c_1=p_1,p^{-1}_2$ and $c_2=p_1,p^{-1}_3$ and $c_3=p_2,p^{-1}_3$ must be $0$. But the sum of values of $g$ and $f$ along $c_1$ plus the sum of values of $g$ and $f$ along $c_2$ minus the sum of values of $g$ and $f$ along $c_3$ is equal to $g(v)+g(w)+2a$, where $a$ is the sum of values of $f$ on all edges of $p_1$ plus the sum of values of $g$ on all inner vertices of $p_1$. Hence, $g(v)-g(w)=g(v)+g(w)-2g(w)=-2a-2g(w)$ is an element of $2A$.
\end{proof}
\begin{lem} \label{five.two} For any $v\in V$ and $a\in A$, the function $g:V\rightarrow A$, such that $g(v)=2a$ and that $g(w)=0$ for all other vertices $w\in V$, is balanceable.
\end{lem}
\begin{proof} We define $f:E\rightarrow A$ by: $f(e)=-2a$ if $e$ is a loop edge adjacent to $v$; $f(e)=-a$ if $e$ is a non-loop edge adjacent to $v$;  $f(e)=0$ if $e$ is not adjacent to $v$. Any cycle of $G$ will either neither contain $v$ nor any of its adjacent edges, or will contain $v$ two of its adjacent non-loop edges, or will contain $v$ and its adjacent loop edge. In all these three cases we get that the sum of values of $g$ and $f$ along the image of that cycle under $\epsilon$ is $0$. Applying Lemma \ref{homologicalcyclesum} yields our lemma.
\end{proof}
\begin{lem} \label{five.three} For any $3$-edge-connected component $T$ of $V$ and any $a\in A$, the function $g:V\rightarrow A$, such that $g(w)=a$ for all $w\in T$ and $g(u)=0$ for all other $u\in V$, is balanceable.
\end{lem}
\begin{proof} Let $\theta_1,\theta_2,...$ be the equivalence classes of all twin $3$-weakly closed ttrails connecting (not necessarily different) vertices from $T$. It follows from Theorem \ref{theorem.twin} that in each class $\theta_i$ we can select some edge $u_i$, which is adjacent to a vertex from $T$ and which is not contained in any other class $\theta_j$. Indeed, Theorem \ref{theorem.twin} guarantees that the first and the last edges belong only to the $3$-weakly closed ttrails from that class.
\\ \\
We define $f:E\rightarrow A$ by $f(e)=-a$ if $e$ is one of the edges $u_1,u_2,...$ and $f(e)=0$ otherwise. From Lemma \ref{weak.space.dec} it follows that every closed ttrail will contain the same amount (could be zero) of vertices $T$ and the edges from $\{u_1,u_2,...\}$. Thus, the sum of values of $g$ and $f$ along any closed ttrail will be $0$, which implies that $g$ is balanceable.
\end{proof}
\begin{thm} \label{balanceable} The Abelian group $B(V,A)$ of all balanceable functions $g:V\rightarrow A$ is isomorphic to $A^{con_3(G)}\times (2A)^{|V|-con_3(G)}$
\end{thm}
\begin{proof} Let $T_1,...,T_{con_3(G)}$ be all equivalence classes of $3$-edge-connected vertices of $V$ and let $w_1,...,w_{con_3(G)}$ be some vertices in these classes. Select and fix some indexing on the rest of the vertices of $V$. We will refer to them as $v_{con_3(G)+1},...,v_{|V|}$. For each vertex $v_i$, where $i>con_3(G)$, let $w(v_i)$ denote one of the vertices $w_1,...,w_{con_3(G)}$, which is $3$-edge-connected to $v_i$.
\\ \\
We define the isomorphism $$\psi:B(V,A)\rightarrow A^{con_3(G)}\times (2A)^{|V|-con_3(G)}$$ by
$$\psi(g)=(g(w_1),...,g(w_{con_3(G)}),g(v_{con_3(G)+1})-g(w(v_{con_3(G)+1})),...,g(v_{|V|})-g(w(v_{|V|})))$$
By Lemma \ref{if.balanceable}, the image of $\psi$ lies in $A^{con_3(G)}\times (2A)^{|V|-con_3(G)}$ and so $\psi$ is well defined.\\ \\
If $\psi(g)=(0,...,0)$ then it is clear that $g(v)=0$ for all the vertices $w_1,...,w_{con_3(G)}$ and also for all the vertices $v_{con_3(G)+1},...,v_{|V|}$. Thus, $\psi$ is injective.\\ \\
For each $(a_1,...,a_{con_3(G)},a''_1,...,a''_{|V|-con_3(G)})\in A^{con_3(G)}\times (2A)^{|V|-con_3(G)}$ we construct $g\in B(V,A)$, such that $$\psi(g)=(a_1,...,a_{con_3(G)},a''_1,...,a''_{|V|-con_3(G)})$$ as follows:
\begin{itemize}
\item For $w_1,...,w_{con_3(G)}$ set $g(w_j)=a_j\,\,$;
\item For $v_{con_3(G)+1},...,v_{|V|}$ set $g(v_i)=a''_{i-con_3(G)}+g(w(v_i))\,$.
\\
\end{itemize}
Lemmas~\ref{five.two} and~\ref{five.three} imply that $g$ is balanceable. A straightforward plugging-in shows that $\psi(g)=(a_1,...,a_{con_3(G)},a''_1,...,a''_{|V|-con_3(G)})$.
\end{proof}
In~\cite{NI} it is proved that for a finite group $A$ and a connected graph $G$, $|W(V\bigcup E,A)|=|A|^{|V|+con_3(G)-1}$. This result now follows from our Lemma~\ref{quotient.group} and Theorems~\ref{balanced.edge} and~\ref{balanceable}, based on the obvious fact that $|A|=|A_2|\cdot |2A|$. In the next section we prove a stronger result in Theorem~\ref{six.three}.
\section{Balanced Functions from the Graph $G$ to an Abelian Group $A$}
Notice that when in this section we speak of the value of a function $h\in W(V\bigcup E,A)$ along a ttrail $p$ from vertex $v$ to vertex $w$, we mean the sum of all values of $h$ on vertices and edges of $p$.
\\ \\
Notice that for any balanced function $h$, any two $2$-edge-connected vertices $v$ and $w$, and any two edge-disjoint paths $p'$ and $p''$ from $v$ to $w$, we must always have $h(p')+h(p'')=h(v)-h(w)$, since $h$ is balanced and $0=h(p')+h(p''^{-1})=h(p)+(h(w)+h(p'')-h(v))$.
\\ \\
Notice that the dimension of the space $\Lambda$ (which was defined in the proof of Theorem~\ref{balanced.edge}) is equal to the total number of vertices of $G$ minus the number of the vertex representatives selected for each $3$-weak connected class and this ``coincides'' with the number of vertices $v_{con_3(G)+1},...,v_{|V|}$, which appeared in the proof of Theorem~\ref{balanceable}. Actually, $\Lambda$ is just a lift back to the $3$-Weak Cycle Space of $G$ of its quotient by its subspace $-$ the Cycle Space of $G$. Obviously, by making selections in the proof of Theorem~\ref{balanced.edge}, we have effectively selected and fixed a specific lift.
\begin{lem} \label{six.one} For every $3$-edge-connected component $\{v_1,...,v_m\}$ of $G$ containing $m$ vertices, we can select some $m-1$ short $3$-weakly closed ttrails $s_1,...,s_{m-1}$ between these vertices so that:
\begin{enumerate}
\item All the vertices $v_1,...,v_m$ are pairwise connected by these $s_1,...,s_{m-1}$ and/or by their sums;
\item For any two distinct $3$-edge-connected vertices $v_i$ and $v_j$, connected by some $s_r$, there are two other ttrails $p_1$ and $p_2$ between $v_i$ and $v_j$, such that $s_r, p_1$ and $p_2$ are pairwise edge-disjoint.
\end{enumerate}
\end{lem}
\begin{proof}
The lemma is clear, when $m=1$ or $m=2$. Assume that the lemma is true for any $3$-edge connected component of a graph containing of up to $m$ vertices. For a $3$-edge connected component $v_1,...,v_{m+1}$ we can always find some short $3$-weakly closed ttrail $s_m$ from $v_{m+1}$ to one of the vertices $v_i$. By Lemma \ref{weak.space.dec} and Theorem \ref{theorem.twin} every ttrail from $v_{m+1}$ to $v_i$, which contains any of the edges (and, hence, of the inner vertices) of $s_m$ must also contain the initial edge $e_1$ of $s_m$. Hence deleting $e_1$ from $G$ will cut all such ttrails. But, by the Menger's Theorem, this deletion must preserve at least two edge-disjoint ttrails $p_1$ and $p_2$ from $v_{m+1}$ to $v_i$. Now we apply the induction hypothesis to obtain short $3$-weakly closed ttrails $s_1,...,s_{m-1}$. Thus, we obtain $s_1,...,s_{m-1},s_m$ which satisfies the requirements of the lemma.
\end{proof}
The following lemma is combining and generalizing results stated in Theorem~\ref{four.one} and Lemmas~\ref{if.balanceable} and~\ref{five.two}.
\begin{lem} \label{six.two} If a function $h:V\bigcup E\rightarrow A$ is balanced then for every $3$-edge connected vertices $v_i$ and $v_j$ and every short $3$-weakly closed ttrail $s_r$ from $v_j$ to $v_i$, $2h(s_r)=h(v_i)-h(v_j)$
\end{lem}
\begin{proof} As in Lemma~\ref{six.one} we find ttrails $p_1$ and $p_2$ from $v_j$ to $v_i$ such that $s_r, p_1$ and $p_2$ are pairwise edge-disjoint. The inverse ttrails $p_1^{-1},p_2^{-1}$ and $s_r^{-1}$ will contain the same edges as the original ones and the same vertices, except containing the initial $v_i$ and not containing the final $v_j$. Hence we can compute the values of the balanced function $h$ along the following the closed ttrails: $h(p_1)+h(v_j)+h(p_2)-h(v_i)=0$, $h(p_1)+h(v_j)+h(s_r)-h(v_i)=0$, and $h(s_r)+h(v_j)+h(p_2)-h(v_i)=0$. Adding the last two equations and subtracting the first equation gives us $2h(s_r)+h(v_j)-h(v_i)=0$.
\end{proof}
Let $s_{1,1},...,s_{1,r_1},...,s_{con_3(G),1},...,s_{con_3(G),r_{con_3(G)}}$, be the selected short $3$-weakly closed ttrails (as in the Lemma \ref{six.one}) from all $3$-edge connected components of $G$. Numbers $r_i$ can be equal to zero (if the corresponding $3$-edge connected component consists of only one vertex) and $r_1+...+r_{con_3(G)}=|V|-con_3(G)$. We label $s'_1=s_{1,1},...,s'_{r_1}=s_{1,r_1},...,s'_{|V|-con_3(G)}=s_{con_3(G),r_{con_3(G)}}$. Clearly, all homological paths $\epsilon(s'_1),...,\epsilon(s'_{|V|-con_3(G)})$ are $3$-weak homological cycles and, thus, belong to the $3$-Weak Cycle Space of $G$. Notice, that no nontrivial $\mathbb{F}_2$-linear combination of these $3$-weak homological cycles belongs to the Cycle Space of $G$. Indeed, since for each equivalence component $W_i$ containing $m$ vertices we selected $m-1$ short $3$-weakly closed ttrails $s_{i,1},...,s_{i,r_i}$ and since for any $w_i\in W_i$, all other vertices of $W_i$ are connected to $w_i$ by these $s_{i,1},...,s_{i,r_i}$ or their sums, the pigeon-hole principle assures that no nontrivial homological cycles can be obtained by taking $\mathbb{F}_2$-linear combinations of $\epsilon(s_{i,1}),...,\epsilon(s_{i,r_i})$. And $\delta$ takes a nontrivial $\mathbb{F}_2$-linear combination of $\epsilon(s_{i,1}),...,\epsilon(s_{i,r_i})$ to a nonempty set of vertices from $W_i$. Hence, $\delta$ of any nontrivial $\mathbb{F}_2$-linear combination of $\epsilon(s'_1),...,\epsilon(s'_{|V|-con_3(G)})$ can not be a homological cycle.
\\ \\
Thus, the subspace $\Lambda'$ of the $3$-Weak Cycle Space of $G$ spanned by $\epsilon(s'_1),...,\epsilon(s'_{|V|-con_3(G)})$ is also a lift to $3$-Weak Cycle Space of $G$ of its quotient by the Cycle Space of $G$.

\begin{thm} \label{six.three} The Abelian group $W(V\bigcup E,A)$ of all balanced functions $h:V\bigcup E\rightarrow A$ is isomorphic to $A^{|V|+con_3(G)-con(G)}$
\end{thm}
\begin{proof} Similarly to what we did in the proof of Theorem~\ref{balanced.edge}, we choose any basis $c_1,...,c_{|E|-|V|+con(G)}$ of the Cycle Space of $G$. Next, we select some edges $e_1,...,e_{con_3(G)-con(G)}$ of $G$ so that they, together with $c_1,...,c_{|E|-|V|+con(G)}$ and with the above-mentioned $\epsilon(s'_1),...,\epsilon(s'_{|V|-con_3(G)})$, constitute a basis of $\mathbb{F}_2^E$. Finally, similarly to what we did in the proof of Theorem~\ref{balanceable}, we select some representative vertices $w_1,...,w_{con_3(G)}$ for the $3$-edge-connected components of $G$.
\\ \\
An isomorphism $\xi:W(V\bigcup E,A)\rightarrow A^{|V|+con_3(G)-con(G)}$ is obtained by defining $\xi(h)$ as:
$$\!\!\!\!\!\!\!\! \xi(h)=(h(w_1),...,h(w_{con_3(G)}),h(e_1),...,h(e_{con_3(G)-con(G)}),h(s'_1),...,h(s'_{|V|-con_3(G)}))$$
To see that $\xi$ is one-to-one observe that if $\xi(h)=0$ then, due to Lemma \ref{six.two}, $h$ must be $0$ on every vertex of $G$. Hence, we can regard $h$ as an element of $H(E,A)$. But $h$ is $0$ on every homological cycle of $G$ and on $\epsilon(s'_1),...,\epsilon(s'_{|V|-con_3(G)})$ and $e_1,...,e_{con_3(G)-con(G)}$. Hence $h$ is $0$ on every element of $\mathbb{F}_2^E$. Thus, $h$ is the zero function on $V\bigcup E$.
\\ \\
To show that $\xi$ is onto we, for every element $\textbf{a}= (a_1,...,a_{|V|+con_3(G)-con(G)})$ of $A^{|V|+con_3(G)-con(G)}$, construct $h\in W(V\bigcup E,A)$ such that $\xi(h)=\textbf{a}$.
\begin{enumerate}
\item For all the vertices $w_1,...,w_{con_3(G)}$ we set $h(w_i)=a_i$;
\item For all the edges $e_1,...,e_{con_3(G)-con(G)}$ we set $h(e_t)=a_{t+con_3(G)}$;
\item For all the ttrails $s'_1,...,s'_{|V|-con_3(G)}$ we require $h(s'_j)=a_{j+2con_3(G)-con(G)}$;
\item We set $h(c_i)=0$ for all $c_1,...,c_{|E|-|V|+con(G)}$.
\end{enumerate}
Next, we, starting from $w_1,...,w_{con_3(G)}$, inductively extend $h$ to all remaining vertices $v_{1+con_3(G)},...,v_{|V|}$ of $G$ as follows: whenever $h$ is defined on some vertex $u$ and not yet defined on a vertex $v$ such that there is a short $3$-weakly closed ttrail $s'_i$ from $v$ to $u$, we set $h(v)=h(u)+2h(s_i)$. After obtaining values of $h$ on all the vertices, we compute values of $h$ on $\epsilon(s'_1),...,\epsilon(s'_{|V|-con_3(G)})$ by subtracting values on appropriate vertices from values on $s'_1,...,s'_{|V|-con_3(G)}$.
\\ \\
Finally, we extend $h$ by $\mathbb{F}_2$-linearity from the basis $$c_1,...,c_{|E|-|V|+con(G)},e_1,...,e_{con_3(G)-con(G)},\epsilon(s'_1),...,\epsilon(s'_{|V|-con_3(G)})$$ of $\mathbb{F}_2^E$ to the enire $\mathbb{F}_2^E$. Thus, we obtain values of $h$ on all the edges of $G$.  By a direct plugging-in we can see that indeed $\xi(h)=\textbf{a}$
\end{proof}

\section{Conclusion} When every element of $A$ is of the rank $2$, for example if $A=\mathbb{Z}_2$, then $2A=0$ and $A_2=A$. In that case there is no difference between the weights on edges and the gains on edges. Thus our results for $H(E,A)$ in that case coincide for that case with the classical results for the balanced voltage and gain graphs. Graphs, with the elements of $\mathbb{Z}_2$ assigned to their edges, are called signed graphs. Balanced signed graphs were first characterized by Harary. When $A$ is the additive group of the real numbers then $2A=A$ and $A_2=0$ and our results for $H(E,A)$ coincide with the results of~\cite{OI}. In the case when $A$ is a finite group, which is the case studied in~\cite{NI}, Theorem 4 of~\cite{NI} follows from our Theorem~\ref{six.three}.
\\ \\
In this work we studied the structure of the group of balanced functions $W(V\bigcup E,A)$, its subgroup $H(E,A)$, and its factor-group $B(V,A)$. The dual problem is to understand the group structure of the subgroup $B'(V,A)$ of $W(V\bigcup E,A)$ of all balanced functions on vertices, and the group structure of the factor-group $H'(E,A)$ of all balanceable functions on edges. Thus the elements of $B'(V,A)$ are, by abuse of notation, such elements of $W(V\bigcup E,A)$, which take $0$ value of every edge of $G$. For the case of $A=Z_2$ this would be identical to describing the consistent marked (vertex-signed) graphs, which were studied and characterized in~\cite{Acharya1} , \cite{Acharya2}, \cite{Rao}, \cite{Hoede}, \cite{RX} and in~\cite{ZR}. In~\cite{Roberts} some relations between consistent marked graphs and balanced signed graphs are studied.
\\ \\
The structure of the group of balanced Abelian group labeling on other discrete structures, such as simplicial complexes and hypergraphs, is also of interest.


\begin{thebibliography}{1}
\bibitem{Acharya1} B.D. Acharya. A characterization of consistent marked graphs. Nat. Acad. Sci Lett. 6 (1983) 431-440.
\bibitem{Acharya2} B.D. Acharya. Some further properties of consistent marked graphs. Indian J. Pure Appl. Math. 15 (1984) 837-842.
\bibitem{OI} R. Balakrishnan and N. Sudharsanam. Cycle vanishing edge valuations of a graph. Indian J. Pure Appl. Math. 13(3) (1982) 313-316.
\bibitem{Diestel} R. Diestel. Graph Theory. Springer; 4th edition 2010. Corrected 2nd printing 2012 edition.
\bibitem{GT} J. L. Gross and T. W. Tucker. Topological Graph Theory. Dover Publications; 1987.
\bibitem{Harju} T. Harju. Lecture Notes on Graph Theory. Department of Mathematics, University of Turku, Finland (1994-2012). URL: $http://users.utu.fi/harju/graphtheory/graphtheory.pdf$
\bibitem{Hoede} C. Hoede. A characterization of consistent marked graphs. J. Graph Theory 16 (1992) 17-23.
\bibitem{NI} M. Joglekar, N. Shah and A. A. Diwan. Balanced group-labeled graphs. Discrete Mathematics 312 (2012) 1542-1549.
\bibitem{Menger} K. Menger. "Zur allgemeinen Kurventheorie". Fund. Math. 10 (1927) 96-115.
\bibitem{Rao} S. B. Rao. Characterization of harmonious marked graphs and consistent nets. J. Combin. Inform. System Sci. 9 (1984) 97-112.
\bibitem{Roberts} F. S. Roberts. On balanced signed graphs and consistent marked graphs. Electronic Notes in Discrete Mathematics
Volume 2 (April 1999) 94-105.
\bibitem{RX} F. S. Roberts and S. Xu. Characterizations of consistent marked graphs. Discrete Appl. Math 127 (2003) 357-371.
\bibitem{RZ} K. Rybnikov and T. Zaslavsky. Criteria for Balance in Abelian Gain Graphs, with Applications to Piecewise-Linear Geometry. Discrete \& Computational Geometry. Volume 34, Issue 2 (August 2005) 251-268.
\bibitem{Tsin} Y. H. Tsin. Yet another optimal algorithm for 3-edge-connectivity. Journal of Discrete Algorithms 7.1 (2009) 130-146.
\bibitem{Zaslavsky} T. Zaslavsky. Bibliography of Signed and Gain Graphs. The Electronic Journal of Combinatorics. Dynamical Surveys. DS8 (September 26, 1999).
\bibitem{ZR} T. Zaslavsky. Consistent Vertex-Signed Graphs. A lecture at the C.R. Rao Advanced Institute of Mathematics, Statistics and Computer Science, Univ. of Hyderabad, India (28 July 2010). URL $http://www.math.binghamton.edu/zaslav/Tpapers/consistent-v-s.slides.20100728.pdf$
\end{thebibliography}
\end{document}